\documentclass[oneside,british,english]{amsart}
\usepackage[utf8]{inputenc}
\usepackage{babel}
\usepackage{amsbsy}
\usepackage{amstext}
\usepackage{amsthm}
\usepackage{amssymb}
\usepackage[letterpaper]{geometry}
\usepackage[bookmarks=false,
 breaklinks=false,pdfborder={0 0 1},backref=section,colorlinks=false]
 {hyperref}

\makeatletter
\numberwithin{equation}{section}
\numberwithin{figure}{section}
\theoremstyle{plain}
\newtheorem{thm}{\protect\theoremname}[section]
\theoremstyle{definition}
\newtheorem{defn}[thm]{\protect\definitionname}
\theoremstyle{remark}
\newtheorem{rem}[thm]{\protect\remarkname}
\theoremstyle{plain}
\newtheorem{prop}[thm]{\protect\propositionname}
\theoremstyle{remark}
\newtheorem*{claim*}{\protect\claimname}
\theoremstyle{plain}
\newtheorem{cor}[thm]{\protect\corollaryname}
\theoremstyle{remark}
\newtheorem{notation}[thm]{\protect\notationname}
\theoremstyle{definition}
\newtheorem*{example*}{\protect\examplename}
\theoremstyle{plain}
\newtheorem{lem}[thm]{\protect\lemmaname}

\usepackage{babel}

\numberwithin{equation}{section}

\providecommand{\corollaryname}{Corollary}
\providecommand{\propositionname}{Proposition}
\providecommand{\theoremname}{Theorem}

\addto\captionsbritish{\renewcommand{\corollaryname}{Corollary}}
\addto\captionsbritish{\renewcommand{\definitionname}{Definition}}
\addto\captionsbritish{\renewcommand{\lemmaname}{Lemma}}
\addto\captionsbritish{\renewcommand{\propositionname}{Proposition}}
\addto\captionsbritish{\renewcommand{\remarkname}{Remark}}
\addto\captionsbritish{\renewcommand{\theoremname}{Theorem}}
\addto\captionsenglish{\renewcommand{\corollaryname}{Corollary}}
\addto\captionsenglish{\renewcommand{\definitionname}{Definition}}
\addto\captionsenglish{\renewcommand{\lemmaname}{Lemma}}
\addto\captionsenglish{\renewcommand{\propositionname}{Proposition}}
\addto\captionsenglish{\renewcommand{\remarkname}{Remark}}
\addto\captionsenglish{\renewcommand{\theoremname}{Theorem}}

\providecommand{\definitionname}{Definition}
\providecommand{\lemmaname}{Lemma}

\providecommand{\remarkname}{Remark}

\newcommand{\Z}{\mathbb{Z} }

\makeatother

\addto\captionsbritish{\renewcommand{\claimname}{Claim}}
\addto\captionsbritish{\renewcommand{\corollaryname}{Corollary}}
\addto\captionsbritish{\renewcommand{\definitionname}{Definition}}
\addto\captionsbritish{\renewcommand{\examplename}{Example}}
\addto\captionsbritish{\renewcommand{\lemmaname}{Lemma}}
\addto\captionsbritish{\renewcommand{\notationname}{Notation}}
\addto\captionsbritish{\renewcommand{\propositionname}{Proposition}}
\addto\captionsbritish{\renewcommand{\remarkname}{Remark}}
\addto\captionsbritish{\renewcommand{\theoremname}{Theorem}}
\addto\captionsenglish{\renewcommand{\claimname}{Claim}}
\addto\captionsenglish{\renewcommand{\corollaryname}{Corollary}}
\addto\captionsenglish{\renewcommand{\definitionname}{Definition}}
\addto\captionsenglish{\renewcommand{\examplename}{Example}}
\addto\captionsenglish{\renewcommand{\lemmaname}{Lemma}}
\addto\captionsenglish{\renewcommand{\notationname}{Notation}}
\addto\captionsenglish{\renewcommand{\propositionname}{Proposition}}
\addto\captionsenglish{\renewcommand{\remarkname}{Remark}}
\addto\captionsenglish{\renewcommand{\theoremname}{Theorem}}
\providecommand{\claimname}{Claim}
\providecommand{\corollaryname}{Corollary}
\providecommand{\definitionname}{Definition}
\providecommand{\examplename}{Example}
\providecommand{\lemmaname}{Lemma}
\providecommand{\notationname}{Notation}
\providecommand{\propositionname}{Proposition}
\providecommand{\remarkname}{Remark}
\providecommand{\theoremname}{Theorem}

\begin{document}
\selectlanguage{british}%
\global\long\def\Hom{{\rm Hom}}%
\global\long\def\End{{\rm End}}%
\global\long\def\Aut{{\rm Aut}}%
\global\long\def\Hol{{\rm Hol}}%
\global\long\def\Ann{{\rm Ann}}%
\global\long\def\Soc{{\rm Soc}}%
\global\long\def\Fix{{\rm Fix}}%
\global\long\def\bbZ{{\rm \mathbb{Z}}}%

\title{\selectlanguage{british}%
On Modular maximal-cyclic braces}
\author{\selectlanguage{british}%
Arpan Das}
\address{\selectlanguage{british}%
Harish-Chandra Research Institute, A CI of Homi Bhabha National Institute,
Chhatnag Road, Jhunsi, Prayagraj-211 019, India}
\email{\selectlanguage{british}%
arpandas@hri.res.in}
\author{\selectlanguage{british}%
Arpan Kanrar}
\address{\selectlanguage{british}%
Harish-Chandra Research Institute, A CI of Homi Bhabha National Institute,
Chhatnag Road, Jhunsi, Prayagraj-211 019, India}
\email{\selectlanguage{british}%
arpankanrar@hri.res.in}
\keywords{\selectlanguage{british}%
Braces; Holomorphs}
\subjclass[2020]{\selectlanguage{british}%
16T25, 20H25,20K30,20D15}
\selectlanguage{english}%
\begin{abstract}
Inspired by a conjecture by Guarnieri and Vendramin concerning the
number of braces with a generalized quaternion adjoint group, many
researchers have studied braces whose adjoint group is a non-abelian
$2$-group with a cyclic subgroup of index $2$. Following this direction,
braces with generalized quaternion, dihedral, and semidihedral adjoint
groups have been classified. It was found that the number of such
braces stabilizes as the group order increases. In this paper, we
consider the remaining open case of modular maximal-cyclic groups.
We show that these braces possess only one non-cyclic additive
group structure, and, in contrast to previous findings, the number
of such braces increases with increasing order. 
\end{abstract}

\maketitle

\section{Introduction}

To study non-degenerate involutive set-theoretic solutions to the
Yang-Baxter equation, Rump \cite{Rump07} introduced the algebraic
structure called a \emph{brace}. Finite braces can be viewed as groups
$G$ with an affine structure, which is a bijective $1$-cocycle onto
a right $G$-module. A key invariant of finite involutive solutions
is the \emph{Involutive Yang-Baxter group} \cite{ESS99,IYB-grp}, which
is itself a brace. Brace theory has been used to study several key
properties of involutive solutions, including decomposability, multi-permutation,
and simplicity \cite{ESS99,BBMS,COK,mpl2,MSC23}. Moreover, braces
arise in the theory of regular affine groups \cite{Caranti2006,Cedo2014,Catino2016},
Hopf-Galois structures \cite{Angiono2017,Fea2012,Childs2013}, and
various other related topics.

The non-abelian groups of order $2^{n}$ ($n\ge5$) with a cyclic
subgroup of index $2$ are of four types \cite{hall73,zass58}, namely
the dihedral group $D_{2^{n}}$, the generalized quaternion group
$Q_{2^{n}}$, the semidihedral group $SD_{2^{n}}$, and the modular
maximal-cyclic group $M_{2^{n}}$ given by the presentations 
\begin{align*}
D_{2^{n}} & =\langle a,b~|~a^{2^{n-1}}=b^{2}=1,~bab^{-1}=a^{-1}\rangle\\
Q_{2^{n}} & =\langle a,b~|~a^{2^{n-1}}=1,~b^{2}=a^{2^{n-2}},~bab^{-1}=a^{-1}\rangle\\
SD_{2^{n}} & =\langle a,b~|~a^{2^{n-1}}=b^{2}=1,~bab^{-1}=a^{-1+2^{n-2}}\rangle\\
M_{2^{n}} & =\langle a,b~|~a^{2^{n-1}}=b^{2}=1,~bab^{-1}=a^{1+2^{n-2}}\rangle.
\end{align*}

\begin{defn}
A set $A$ equipped with two group structures $\circ$ and $+$ is said to be a {\em right brace} if
\begin{enumerate}
\item $(A,+)$ is an abelian group; 
\item the relation $(b+c)\circ a+a=b\circ a+c\circ a$ holds for all $a,b,c\in A$. 
\end{enumerate}
\end{defn}

\noindent The group $(A,\circ)$ is called \emph{the adjoint group}
of the brace $A$. For convenience, we will often omit the word `right’,
referring to `right braces’ simply as `braces’.

\vspace{0.2cm}

In a foundational work, Guarnieri and Vendramin \cite{GV17} extended
the concept of braces to skew braces by relaxing the commutative property
of the additive group. Their paper included computational results
for counting braces and skew braces of small order, which led them
to propose a conjecture about the number of braces with a generalized
quaternion adjoint group of order $4m$. This conjecture was subsequently
proven by Rump \cite{Rump20} for braces of 2-power size and later
in full generality by Byott and Ferri \cite{Byott}. The latter also
classified $2$-power dihedral braces.

These findings have inspired a new line of research focused on classifying
braces whose adjoint group is a non-abelian $2$-group containing
a cyclic subgroup of index $2$. As part of this effort, semidihedral
braces were recently classified in \cite{JAlg}. It's interesting
to note that for dihedral ($D_{2^{n}}$), generalized quaternion ($Q_{2^{n}}$),
and semidihedral ($SD_{2^{n}}$) adjoint groups, the number of corresponding
braces stabilizes as the order increases. Furthermore, conjectures
regarding the number of braces of orders $8p$ and $12p$ for large
primes $p$, based on extensive computations by Bardakov, Neshchadim,
and Yadav \cite{Bardakov2020}, have been recently confirmed by Crespo,
Gil-Mu\~noz, Rio, and Vela \cite{Crespo2023,Crespo2023a}.

In this paper, we address the final remaining case: braces where the
adjoint group is a modular maximal-cyclic group. Utilizing the methods
developed in \cite{Crespo2023,Crespo2023a,JAlg,Rump20,Byott}, we
demonstrate that these braces have only one possible non-cyclic additive
group structure (Theorem \ref{thm:additive}). We then construct all
possible brace structures (Theorem \ref{thm:1}) and show that their
number increases with increasing order (Corollary \ref{cor:final}),
a surprising contrast to the stabilization observed in the other cases. 

\section{Basics on braces}

In this section, we provide basic definitions and results on braces,
and introduce notations used throughout the paper.\\
 For a brace $A$, set 
\[
a^{b}:=a\circ b-b~~\text{ and }~~\sigma(a)(b):=b^{a^{-1}}
\]
for $a,~b\in A$. It turns out that $\sigma(a)\in\Aut(A,+)$ for all
$a\in A$, and the relation $(a+b)^{c}=a^{c}+b^{c}$ holds in $A$.
\begin{rem}
For a brace $(A,+,\circ)$, the identity element of both the groups
is the same and the inverse of an element $a$ w.r.t. $\circ$ is
$-\sigma(a)(a)$. 
\end{rem}

There are two other equivalent ways of thinking about right braces
that are relevant to this article.
\begin{rem}
\label{rem: braces as 1-cocycles and regular subgroups of Hol} Given
a right brace $(A,+,\circ)$ we define the map 
\[
\rho:(A,\circ)\to\Aut(A,+)
\]
by 
\[
a\mapsto\rho_{a}:=[b\mapsto b^{a}]=\sigma(a)^{-1}.
\]
This makes $\rho$ an anti-homomorphism of groups since 
\[
\rho_{a\circ b}(c)=c^{a\circ b}=(c^{a})^{b}=\rho_{b}(\rho_{a}(c))
\]
for all $a,b,c\in A$. One can check that the condition $a^{b\circ c}=(a^{b})^{c}$
is equivalent to $a\circ(b\circ c)=(a\circ b)\circ c$. Consequently,
we have a right linear group action $(A,+)\curvearrowleft(A,\circ)$
defined as 
\[
a\ast b:=\rho_{b}(a)\,\,\,\text{for all }a,b\in A.
\]

Now recall that for a right linear action of a group $G$ on an Abelian
group $(A,+)$ a map $f:G\to A$ is called a \emph{right 1-cocyle}
if 
\[
f(gh)=f(g)\ast h+f(h)\,\,\,\text{for all }g,h\in G.
\]
So given a right brace $(A,+,\circ)$ we can easily check that the
identity map $1_{A}:(A,\circ)\to(A,+)$ is a bijective right 1-cocycle.
Conversely, given a group $(H,\circ)$ acting linearly on the right
of an abelian group $(A,+)$, and a bijective right 1-cocycle $\pi:H\to A$
(i.e. satisfying $\pi(g\circ h)=\pi(g)\ast h+\pi(h)$ for all $g,h\in H$)
we can define an addition on $H$ by 
\[
g+h:=\pi^{-1}(\pi(g)+\pi(h))\,\,\,\text{for all }g,h\in H.
\]
Then we can check that $(H,+,\circ)$ is a right brace. Therefore
right braces correspond to bijective right 1-cocycles.

Another way to view right braces are as regular subgroups of holomorph
groups. Recall that for a group $G$ we define the \emph{holomorph
of $G$} to be 
\[
\Hol(G):=G\rtimes\Aut(G).
\]
Then a \emph{regular subgroup of }$\Hol(G)$ is defined to be a subgroup
of the form 
\[
\{(g,\varphi(g))\in\Hol(G)\,|\,g\in G\}
\]
for some set function $\varphi:G\to\Aut(G)$. Now given an abelian
group $(A,+)$ let 
\[
\mathcal{B}(A):=\{(A,+,\circ)\,\,|\,\,(A,+,\circ)\text{ is a right brace}\}
\]
and let 
\[
\mathcal{S}(A):=\{G\,\,|\,\,G\text{ is a regular subgroup of }\Hol(A,+)\}.
\]
Then the map 
\[
[(A,+,\circ)\mapsto\{(a,\rho_{a})\in\Hol(A,+)\,\,|\,\,a\in A\}]:\mathcal{B}(A)\to\mathcal{S}(A)
\]
is a bijection. Hence, right braces can be viewed as regular subgroups
of holomorphs of abelian groups. 
\end{rem}

\begin{defn}
For a right brace $A$, a subgroup $I$ of the additive group $(A,+)$
is called a \emph{right ideal}, if it is stable under the action $\rho$,
that is, $\rho_{a}(x)\in I$ for all $a\in A,~x\in I$. It turns out
that a right ideal is also a subgroup of the adjoint group. A right
ideal is said to be an \emph{ideal} if it is a normal subgroup of
$(A,\circ)$. 
\end{defn}
\begin{defn}
Two braces $(A_1,+_1,,\circ_1)$ and $(A_2,+_2,,\circ_2)$ are said to be {\em isomorphic} if there is a bijective map $f:A_1\to A_2$ satisfying 
    \begin{align*}
        f(x+_1y) &=f(x)+_2f(y)\\
        f(x\circ_1 y) &=f(x)\circ_2f(y)
    \end{align*}
    for all $x,y\in A_1$.
\end{defn}

Given a right brace $(A,+,\circ)$ we define its \emph{socle }to be
$\ker\rho$, that is, 
\begin{align*}
\Soc(A):= & \{a\in A\,\,|\,\,b^{a}=b\,\,\text{for all }b\in A\}=\{a\in A\,\,|\,\,\sigma(a)=1_{A}\}.
\end{align*}
In \cite{Rump07}, it is proved that $\Soc(A)$ is an ideal of the
brace $A$. 

\section{The additive structures of modular maximal-cyclic braces}

A \emph{modular maximal-cyclic {\em(MMC)} brace} is defined as a brace $A$ whose
adjoint group is isomorphic to $M_{2^{m+2}}$ for some integer $m$.
In this section, we obtain that the socle of a modular maximal-cyclic
brace $A$ of size $2^{m+2}$ ($m\ge3$) is non-trivial, and we determine
the possible additive structures of the brace $A$.\\
 We recall the presentation of $M_{2^{m+2}}$ ($m\ge3$) given in
the Introduction section 
\begin{equation}
M_{2^{m+2}}=\langle a,b~|~a^{2^{m+1}}=b^{2}=1,~bab^{-1}=a^{1+2^{m}}\rangle.\label{p1}
\end{equation}

We first classify all normal subgroups of a modular-maximal cyclic
group. 
\begin{prop}
\label{prop:subgrps} Let $M_{2^{m+2}}$ be modular maximal-cyclic
group with presentation \eqref{p1} and $H$ be a non-trivial proper
subgroup of $M_{2^{m+2}}$. Then, $H$ is one of the following forms
$(1\le s\le m)$ 
\begin{itemize}
\item $H\subseteq\langle a\rangle;$ 
\item $H=\langle b\rangle$;
\item $H=\langle ba\rangle;$ 
\item $H=\langle ba^{2^{s}}\rangle;$ 
\item $H=\langle b,a^{2^{s}}\rangle.$ 
\end{itemize}
Moreover every non-trivial subgroup except $\langle ba^{2^{m}}\rangle$
and $\langle b\rangle$ is normal and contains $\langle a^{2^{m}}\rangle$. 
\end{prop}

\begin{proof}
When we say $ba^{n}\in M_{2^{m+2}}$ we mean $0\le n\le2^{m+1}-1$.
For odd $i$, we have 
\begin{equation}
\langle ba^{i}\rangle=\{1,ba,ba^{3},\ldots,ba^{2^{m+1}-1},a^{2},a^{4},\ldots,a^{2^{m+1}-2}\}\label{eq:1}
\end{equation}
and for even $i$, $(ba^{i})^{2}=a^{2i}$. Thus every non-trivial
subgroup of $M_{2^{m+2}}$ except $\langle ba^{2^{m}}\rangle$ and
$\langle b\rangle$ and contains $\langle a^{2^{m}}\rangle$. Observe
that $a^{i}(ba^{r})a^{-i}=(ba^{r})(a^{2^{m}})^{i}~\text{ and }~ba^{i}(ba^{r})(ba^{i})^{-1}=(ba^{r})(a^{2^{m}})^{r-i}.$
Thus all the mentioned subgroups except those two are normal.

Now let $H$ be a non-trivial proper subgroup which is neither of
the first three forms. We first do some reductions. At first, note
that if $ba^{i}\in H$ then $i$ is even, since otherwise for odd
$i$ we have $H=\langle ba^{i}\rangle=\langle ba\rangle$ or $H=M_{2^{m+2}}$
as $\langle ba\rangle$ has index two. Again, since $H\neq\langle b\rangle$
we know that there is some minimal positive even $i_{0}$ for which
$ba^{i_{0}}\in H$. We write this $i_{0}$ in the form $2^{r}q$ with
$r\geq1$ and $q$ odd. Then $ba^{2^{r}qq^{\prime}}\in H$ for all
odd $q^{\prime}.$ We can choose $q^{\prime}$ such that $qq^{\prime}\equiv1\,{\rm mod}\,2^{m+1}$.
Therefore, $i_{0}=2^{r}$ with $r\geq1$ minimal. In fact, by this
argument we have also proved that if $ba^{2^{l}q}\in H$ (with $q$
odd), then $ba^{2^{l}}\in H$ and so $l\geq r$. Now the following
claim settles the proposition.
\begin{claim*}
Let $H\neq1,M_{2^{m+2}},\langle b\rangle,\langle ba\rangle$ and not
a subgroup of $\langle a\rangle$, and let $r\geq1$ be minimal such
that $ba^{2^{r}}\in H$. We have :
\begin{itemize}
\item If $ba^{j}\in H$ implies $j=2^{r}q$ with $q$ odd, then $H=\langle ba^{2^{r}}\rangle$.
\item If $ba^{j}\in H$ for some $j=2^{l}q$ with $l>r$ and $q$ odd, then
$b\in H$. Further, if $a^{2^{s}}\in H$ with $s\geq0$ minimal, then
we have $s\geq1$ and $H=\langle b,a^{2^{s}}\rangle.$
\end{itemize}
\end{claim*}
\begin{proof}[Proof of Claim : ]
For the first item, it is enough to show $H\cap\langle a\rangle=\langle a^{2^{r+1}}\rangle$.
Otherwise if $a^{2^{k}}\in H$ for $k<r+1$ then $ba^{2^{k}(2^{r-k}+1)}\in H$.
So if $k<r$, then by the same argument in the previous paragraph,
we conclude $ba^{2^{k}}\in H$, contradicting minimality of $r$.
Hence $k=r$, but then $ba^{2^{r+1}}\in H$, which again contradicts
the hypothesis that $ba^{j}\in H$ implies $j=2^{r}q$ with $q$ odd.

Next, suppose $ba^{j}\in H$ for some $j=2^{l}q$ with $l>r$ and
$q$ odd. At first, note that odd powers of $a$ can not be in $H.$
Since otherwise $ba^{2^{r}}a^{{\rm odd}}=ba^{2^{r}+{\rm odd}}\in H,$
and so $H=\langle ba\rangle$ or $M_{2^{m+2}}$ which is avoided in
our hypothesis. So we assume $a^{2^{s}}\in H$ with $s$ minimal.
Then $s>0$. Now since $ba^{2^{r}}\in H$ we have $a^{2^{r+1}}=(ba^{2^{r}})^{2}\in H$
and so $a^{2^{l}}\in\langle a^{2^{r+1}}\rangle\subset H$. Also $ba^{2^{l}q}\in H$
(with $q$ odd) implies $ba^{2^{l}}\in H$ and so $b\in H$. Hence
$\langle b,a^{2^{s}}\rangle\subset H$. Now suppose $ba^{2^{l}}\in H$
with $l>r$. Then $bba^{2^{l}}=a^{2^{l}}\in H,$ and my minimality
of $s$ we have $l\geq s$. Hence $a^{2^{l}}\in\langle a^{2^{s}}\rangle\subset\langle b,a^{2^{s}}\rangle$
and so $ba^{2^{l}}\in\langle b,a^{2^{s}}\rangle$. Therefore, $H=\langle b,a^{2^{s}}\rangle$.
\end{proof}
The proof of the proposition is now complete.
\end{proof}
\begin{cor}
\label{cor:cor 2.2} Let $A$ be a MMC brace of
size $2^{m+2}$ for $m\ge3$, where $(A,\circ)$ is given by the presentation
\eqref{p1}, then $\{1,a^{2^{m}}\}\subseteq\Soc(A)$. 
\end{cor}

\begin{proof}
Suppose $\Soc(A)$ is trivial. Then the adjoint group $A_{\circ}$
embeds in $\Aut(A,+)$. This provides an automorphism of order $2^{m+1}$,
which is impossible by \cite{Berko}. Hence $\Soc(A)$ is a non-trivial
normal subgroup of $(A,\circ)$. Now we show that every normal subgroup
$H$ of $(A,\circ)$ contains $\{1,a^{2^{m}}\}$. The required result
follows from Proposition \ref{prop:subgrps}. 
\end{proof}
The proof of the following result goes along the same lines as the
proof of \cite[Theorem 3.4]{Byott}. 
\begin{prop}
\label{prop: all_possible_additive_groups}Let $A$ be a MMC
brace of size $2^{m+2}$ for $m\ge2$. Then the additive group $(A,+)$
must be one of the following groups: 
\begin{itemize}
\item $\bbZ/2^{m+2};$ 
\item $\bbZ/2\times\bbZ/2^{m+1};$ 
\item $\bbZ/4\times\bbZ/2^{m};$ 
\item $\bbZ/2\times\bbZ/2\times\bbZ/2^{m};$ 
\item $\bbZ/2\times\bbZ/2\times\bbZ/2\times\bbZ/2^{m-1}.$ 
\end{itemize}
\end{prop}

\begin{proof}
By Remark \ref{rem: braces as 1-cocycles and regular subgroups of Hol}, we can assume $(A,\circ)$ be a
regular subgroup of $\Hol(A)$. So $\Hol(A)$ contains an element
of order $2^{m+1}$. Let rank and exponent of $(A,+)$ be $r$ and
$2^{d}$ respectively. Then by \cite[Lemma 2.6]{Byott}, 
\[
m+1<\left\lceil \log_{2}(r+1)\right\rceil +d.
\]
Since $A_{+}$ has a cyclic factor $H$ of order $2^{d}$, we have
$r-1\le\text{rank of }A_{+}/H\le m+2-d$. Thus 
\[
r-1\le m+1-d+1<\left\lceil \log_{2}(r+1)\right\rceil +1.
\]
Hence $r\le4$. Therefore, if $r=2$ then $m+2-d=1$ or $2$; if $r=3$
then $m+2-d=2$; if $r=4$ then $m+2-d=3$. This gives the listed
possible group structures of $(A,+)$. 
\end{proof}

\section{On automorphisms of some abelian 2-groups}

In this section we recall how to think of elements of $\Aut(N)$,
for an abelian $p$-group $N$, in terms of matrices having modular
entries. We start by taking $p$ to be any prime while recalling the
general results and then eventually specialize to the case $p=2$.
We also write a cyclic group of order $\ell$ additively as $\bbZ/\ell$
.

So let $p$ be any prime and let $N$ be an abelian $p$-group of
the form $\bbZ/p^{e_{1}}\times\cdots\times\bbZ/p^{e_{n}}$ where $e_{1}\leq\cdots\leq e_{n}$
are positive integers. We define 
\[
R_{p}:=\left\{ (a_{ij})_{1\leq i,j\leq n}\in\bbZ^{n\times n}\,\,|\,\,p^{e_{i}-e_{j}}\text{\,divides\,}a_{ij}\text{\,for all\,}i\geq j\right\} .
\]
Noting that any matrix in $R_{p}$ can be written as $PBP^{-1}$ for
some $B\in\bbZ^{n\times n}$ and $P={\rm diag}(p^{e_{1}},\dots,p^{e_{n}})$,
we easily conclude that, under usual matrix multiplication and addition $R_{p}$ forms a ring.

Next we take $\pi_{i}:\bbZ\to\bbZ/p^{e_{i}}$ to be the canonical
projection and $\pi:\bbZ^{n}\to N$ as $(x_{1},\dots,x_{n})^{\intercal}\mapsto(\pi_{1}(x_{1}),\dots,\pi_{n}(x_{n}))^{\intercal}$.
Now we define the map 
\[
\psi:R_{p}\to{\rm End}(N)\,\,,\,\,\,\,\,\,\,U\mapsto\psi(B):=\left[\pi(x_{1},\dots,x_{n})^{\intercal}\mapsto\pi\left(B(x_{1},\dots,x_{n})^{\intercal}\right)\right].
\]
By \cite[Theorem 3.3]{hillar2006} the map $\psi$ is a surjective
ring homomorphism. Moreover, by Lemma 3.4 (\emph{op. cit.}) we have
\[
{\rm Ker}(\psi)=\left\{ (a_{ij})\in R_{p}\,\,|\,\,p^{e_{i}}\text{\,divides\,}a_{ij}\text{\,for all }i,j\right\} .
\]
Finally, by Theorem 3.6 (\emph{ibid.}) we have for $B\in R_{p}$ 
\[
\psi(B)\in\Aut(N)\iff B({\rm mod\,}p)\in{\rm GL}_{n}(\mathbb{F}_{p}).
\]
Hence, we identify $\End(N)$ with $\bar{R}:=R_{p}/{\rm Ker}(\psi)$
and also it is easy to see that $\bar{R}^{\times}$ can be identified
with 
\[
\left\{ (a_{ij})\,\,|\,\,a_{ij}\in\bbZ/p^{e_{i}}\text{\,for all }i,j;\,p^{e_{i}-e_{j}}|a_{ij}\text{ for all }i\geq j;\,(a_{ij})({\rm mod}\,p)\in{\rm GL}_{n}(\mathbb{F}_{p})\right\} .
\]
Thus we can view elements of $\Aut(N)$ as the above type of matrices.
Note that one can directly check that under usual matrix multiplication
of the coset representatives, the above set forms a ring : more precisely,
if $(a_{ij})$ and $(b_{ij})$ are two such matrices with modular
entries as above, then one can treat $a_{ij}$ and $b_{ij}$ simply
as integers and then compute the $ij$-th entry of the product as
$\sum_{k}a_{ik}b_{kj}({\rm mod}\,p^{e_{i}})$. See also \cite[Section 2]{Byott}.

Let $N$ be an abelian $p$-group as in the previous paragraphs. Now
note that we can present any element of $\Hol(N)=N\rtimes\Aut(N)$
by a matrix of the form 
\[
\left(\begin{array}{cc}
B & v\\
0 & 1
\end{array}\right)
\]
where $B$ is the matrix presentation of an element of $\Aut(N)$
and $v\in N$, written as a column vector. Note that this way of presenting
the elements of $\Hol(N)$ reflects the multiplication law : $(n_{1},\varphi_{1})\rtimes(n_{2},\varphi_{2})=(n_{1}\varphi_{1}(n_{2}),\varphi_{1}\varphi_{2})$.
\begin{notation}
We will use capital letters to emphasize the matrix presentation while
writing the elements of $M_{2^{m+2}}$ when we treat this as a subgroup
of $\Hol(N)$ for some abelian $p$-group.
\end{notation}

From this point, we take $p=2$. We have the following proposition
analogous to \cite[Section 4]{Byott}. 
\begin{prop}
\label{prop:the seven relations}Let $N=\bbZ/2^{e_{1}}\times\cdots\times\bbZ/2^{e_{n}}$
with $1\leq e_{1}\leq\cdots\leq e_{n}$ and $e_{1}+\cdots+e_{n}=m+2$,
let 
\[
M_{2^{m+2}}=\left\langle X,Y\,\big|\,X^{2^{m+1}}=Y^{2}=1\,,\,\,YXY^{-1}=X^{1+2^{m}}\right\rangle 
\]
be a regular subgroup of $\Hol(N)$. We write 
\[
X=\left(\begin{array}{cc}
S & v\\
0 & 1
\end{array}\right),\quad Y=\left(\begin{array}{cc}
T & w\\
0 & 1
\end{array}\right).
\]
Then we have the following relations 
\begin{equation}
\begin{cases}
S^{2^{m}}=I, & (1)\\
T^{2}=I, & (2)\\
TS=ST. & (3)
\end{cases}
\end{equation}
\end{prop}
\begin{proof}
    Let the brace corresponding the given regular subgroup be $A$ (Remark \ref{rem: braces as 1-cocycles and regular subgroups of Hol}), and we assume $(A,\circ)$ has the presentation \eqref{p1}. Thus $\rho_a=S$ and $\rho_b=T$, which gives $S^{2^{m+1}}=T^2=I$ and $TST=S^{1+2^m}$. From Corollary \ref{cor:cor 2.2}, we have $S^{2^m}=I$. This induces the required relations.
\end{proof}

Now we turn to the special types of Abelian $2$-groups obtained in
Proposition \ref{prop: all_possible_additive_groups}. We have the
following proposition, the results of which can be proved by other
means as well which do not depend on the matrix presentations. 
\begin{prop}
\label{prop:Aut(N) for all possible Ns}Let $m\ge2$ and $N$ be one
of the abelian $2$-groups obtained in Proposition \ref{prop: all_possible_additive_groups}.
We have : 
\begin{enumerate}
\item[\textit{\emph{(a). }}] If $N=\bbZ/2^{m+2}$, then $|\Aut(N)|=2^{m+1}$. 
\item[\textit{\emph{(b). }}] If $N=\bbZ/2\times\bbZ/2^{m+1}$, then $|\Aut(N)|=2^{m+2}$. 
\item[\textit{\emph{(c). }}] If $N=\bbZ/4\times\bbZ/2^{m}$, then $|\Aut(N)|=\begin{cases}
2^{m+4} & ,\text{for }m>2\\
2^{5}\times3 & ,\text{for }m=2
\end{cases}.$ 
\item[\textit{\emph{(d). }}] If $N=\bbZ/2\times\bbZ/2\times\bbZ/2^{m}$, then $|\Aut(N)|=2^{m+4}\times3$. 
\item[\textit{\emph{(e). }}] If $N=\bbZ/2\times\bbZ/2\times\bbZ/2\times\bbZ/2^{m-1}$, then $|\Aut(N)|=\begin{cases}
2^{m+7}\times3\times7 & ,\text{for }m>2\\
2^{6}\times3^{2}\times5\times7 & ,\text{for }m=2
\end{cases}.$ 
\end{enumerate}
Moreover, in each of the above cases, the matrices in $\Aut(N)$ that
reduce ${\rm mod}\,2$ to upper unipotents form a Sylow 2-subgroup. 
\end{prop}

\begin{proof}
Item (a) is well known, so we prove the Proposition for items (b)-(e).

For $N=\bbZ/2\times\bbZ/2^{m+1}$ we have 
\begin{align*}
\Aut(N)\simeq & \left\{ \left(\begin{array}{cc}
a & b\\
2^{m}c & d
\end{array}\right)\,\,\Bigg|\,\,a,b\in\bbZ/2;\,c,d\in\bbZ/2^{m+1};\,ad\equiv1({\rm mod}\,2)\right\} \\
= & \left\{ \left(\begin{array}{cc}
1 & b\\
2^{m}c & d
\end{array}\right)\,\,\Bigg|\,\,b\in\bbZ/2,\,2^{m}c\in\{0,2^{m}\},\,d\in(\bbZ/2^{m+1})^{\times}\right\} .
\end{align*}
Hence $|\Aut(N)|=2^{m+2}$, and clearly every element of $\Aut(N)$
reduces to an upper unipotent ${\rm mod}\,2$.

For $N=\bbZ/4\times\bbZ/2^{m}$ we first note that 
\[
\Aut(N)\simeq\left\{ \left(\begin{array}{cc}
a & b\\
2^{m-2}c & d
\end{array}\right)\,\,\Bigg|\,\,a,b\in\bbZ/4;\,\,c,d\in\bbZ/2^{m};\,\,ad-2^{m-2}bc\equiv1({\rm mod}\,2)\right\} .
\]
So we consider at first the case $m=2$. In this case $\Aut(N)\simeq{\rm GL}_{2}(\bbZ/4)$.
Its order is the number of solutions of the modular equation 
\[
ad-bc\equiv1({\rm mod}\,2)
\]
for $a,b,c,d\in\bbZ/4$. This means either $ad$ is odd and $bc$
is even or vice versa. For odd $ad$ there are $4$ possibilities
for $a,d\in\bbZ/4$. By same argument there are $16-4=12$ possibilities
for even $bc$. Hence there are $2\times48=96$ solutions. Therefore
$|\Aut(\bbZ/4\times\bbZ/4)|=2^{5}\times3$. In this subcase, matrices
${\small \left(\begin{array}{cc}
a & b\\
c & d
\end{array}\right)}$ which reduce to upper unipotents are given by the constraints : $a,d=1\text{ or }3,\quad c=0\text{ or }2,\quad\text{and\ensuremath{\quad}}b\in\bbZ/4$.
This has clearly $2^{5}$ many possibilities.

Next consider the case $N=\bbZ/4\times\bbZ/2^{m}$ for $m>2$. We
have 
\begin{align*}
\Aut(N)\simeq & \left\{ \left(\begin{array}{cc}
a & b\\
2^{m-2}c & d
\end{array}\right)\,\,\Bigg|\,\,a,b\in\bbZ/4;\,\,c,d\in\bbZ/2^{m};\,\,ad\equiv1({\rm mod}\,2)\right\} \\
= & \left\{ \left(\begin{array}{cc}
a & b\\
2^{m-2}c & d
\end{array}\right)\,\,\Bigg|\,\,a=1\text{ or }3,\,\,b\in\bbZ/4;\,\,2^{m-2}c\in\{0,2^{m-2},2^{m-1},3\times2^{m-2}\}\text{ mod }2^{m};\,\,d\in(\bbZ/2^{m})^{\times}\right\} .
\end{align*}
Hence $|\Aut(N)|=2^{m+4}$, and every matrix reduces to an upper unipotent.

For $N=\bbZ/2\times\bbZ/2\times\bbZ/2^{m}$ we have 
\begin{align*}
\Aut(N)\simeq & \left\{ \left(\begin{array}{ccc}
a & b & c\\
r & s & t\\
2^{m-1}x & 2^{m-1}y & z
\end{array}\right)\,\,\Bigg|\,\,a,b,c,r,s,t\in\bbZ/2;\,x,y,z\in\bbZ/2^{m};\,\text{and }z(as-br)\equiv1({\rm mod}\,2)\right\} \\
= & \left\{ \left(\begin{array}{ccc}
a & b & c\\
r & s & t\\
2^{m-1}x & 2^{m-1}y & z
\end{array}\right)\,\,\Bigg|\,\,c,t\in\bbZ/2;\,x,y\in\{0,1\};\,z\in(\bbZ/2^{m})^{\times};\,\text{and }\left(\begin{array}{cc}
a & b\\
r & s
\end{array}\right)\in{\rm GL}_{2}(\bbZ/2)\right\} .
\end{align*}
Therefore, $|\Aut(N)|=2^{m+4}\times3.$ And matrices that reduce to
upper unipotents are given by the following constraints : $a=s=1,\quad z\in(\bbZ/2^{m})^{\times},\quad r=0,\quad b,c,t\in\bbZ/2,\quad\text{and }x,y\in\{0,1\}$.
These constraints account for $2^{m+4}$ many possibilities.

Finally consider the case when $N=\bbZ/2\times\bbZ/2\times\bbZ/2\times\bbZ/2^{m-1}$.
Here we have 
\[
\left(\begin{array}{cccc}
a & b & c & d\\
e & f & g & h\\
p & q & r & s\\
2^{m-2}t & 2^{m-2}u & 2^{m-2}v & w
\end{array}\right)\in\Aut(N)
\]
if and only if 
\[
\begin{cases}
a,b,c,d,e,f,g,h,p,q,r,s\in\bbZ/2,\\
t,u,v\in\{0,1\},w\in\bbZ/2^{m-1},\\
\left\Vert \begin{array}{cccc}
a & b & c & d\\
e & f & g & h\\
p & q & r & s\\
2^{m-2}t & 2^{m-2}u & 2^{m-2}v & w
\end{array}\right\Vert \equiv1({\rm mod}\,2).
\end{cases}
\]
Here again we consider first the case when $m=2$. Then clearly $\Aut(N)\simeq{\rm GL}_{4}(\bbZ/2)$
so that $|\Aut(N)|=2^{6}\times3^{2}\times5\times7.$ The matrices
that reduce to upper unipotents are given by the constraints : $a=f=r=w=1,\quad e=p=q=t=u=v=0,\quad\text{and }b,c,d,g,h,s\in\bbZ/2$.
These have $2^{6}$ many solutions.

Next we consider the situation when $m>2$. Then the condition 
\[
w\times\left\Vert \begin{array}{ccc}
a & b & c\\
e & f & g\\
p & q & r
\end{array}\right\Vert \equiv1({\rm mod}\,2)
\]
is equivalent to $w\in(\bbZ/2^{m-1})^{\times}$ and $\left(\begin{array}{ccc}
a & b & c\\
e & f & g\\
p & q & r
\end{array}\right)\in{\rm GL}_{3}(\bbZ/2),$which account for $2^{m+1}\times3\times7$ many choices. For each
of these choices the constraints $d,h,s\in\bbZ/2$ and $t,u,v\in\{0,1\}$
have $2^{6}$ many solutions. Putting these together we have 
\[
|\Aut(N)|=2^{m+7}\times3\times7.
\]
Finally, the matrices that reduce ${\rm mod}\,2$ to upper unipotents
are given by the following constraints 
\[
\begin{cases}
a=f=r=1,\\
w\in(\bbZ/2^{m-1})^{\times},\\
e=p=q=0,\\
t,u,v\in\{0,1\},\\
b,c,d,g,h,s\in\bbZ/2,
\end{cases}
\]
having $2^{m-2}\times2^{3}\times2^{6}=2^{m+7}$ many solutions as
expected.

The proof of the proposition is now complete. 
\end{proof}

\section{When the additive group is not $\protect\bbZ/2\times\protect\bbZ/2^{m+1}$}

In this section, we filter the additive structures of a modular maximal-cyclic
brace. We show that the only possible additive structure of non-cyclic
MMC brace is $\mathbb{Z}/2\times\mathbb{Z}/2^{m+1}$.

The possibility of cyclic structure follows from the following result
proved in \cite[Section 7, Theorem 3]{RumpCyc}. 
\begin{prop}
There is a unique cyclic brace of size $2^{m+2}$ $(m\ge3)$ with
MMC adjoint group. 
\end{prop}

Byott and Ferri \cite[Section 7]{Byott} proved that the holomorph
group $\Hol(\bbZ/4\times\bbZ/2^{m})$ contains no quaternion or dihedral
regular subgroups for $m\ge3$. Their technique, however, yields a
slightly more general result. 
\begin{prop}
Let $G$ be a group of size $2^{m+2}$ $(m\ge3)$ contains an element
of order $2^{m+1}$. Then $G$ can not be a regular subgroup of $\Hol(\bbZ/4\times\bbZ/2^{m})$. 
\end{prop}

\begin{proof}
Suppose $G$ is a regular subgroup of $\Hol(\bbZ/4\times\bbZ/2^{m})$
of size $2^{m+2}$. By the arguments of \cite[Section 7]{Byott},
there is a homomorphism $f:G\to\Hol(\bbZ/4\times\bbZ/8)$ such that
$|\frac{G}{\ker f}|\ge2^{5}$. To the contrary, suppose there is an
element $x\in G$ such that $\mathsf{o}(x)=2^{m+1}$. Let $\mathsf{o}(x\ker f)=2^{r}$,
and consider the subgroup $H=\langle x^{2^{r}}\rangle$ of $\ker f$.
Since $|G/H|=2^{r+1}\ge2^{5}$, we obtain $r\ge4$. Thus $G/\ker f$
contains an element of order $16$, which is not possible by \cite[Lemma 7.1]{Byott}. 
\end{proof}
\begin{cor}
There is no regular subgroup of type $M_{2^{m+2}}$ in $\Hol(\bbZ/4\times\bbZ/2^{m})$
for $m\ge3$.
\end{cor}

The following example confirms the existence of regular subgroups
of type $M_{2^{m+2}}$ in $\Hol(\bbZ/2\times\bbZ/2^{m+1})$.
\begin{example*}
\noindent Consider the groups $A=M_{2^{m+2}}$ and $N=\bbZ/2\times\bbZ/2^{m+1}$
for $m\ge3$. From remark \ref{rem: braces as 1-cocycles and regular subgroups of Hol},
it is enough to provide a bijective map $\gamma:A\to N$ and a right
action $\rho:A\to\Aut(N)$ such that with respect to $\rho$, $\gamma$
becomes a bijective right $1$-cocycle. Consider (with respect to
the presentation \eqref{p1} of $A$) 
\begin{align*}
\rho(a) & =\begin{pmatrix}1 & 0\\
0 & 1
\end{pmatrix},\qquad\rho(b)=\begin{pmatrix}1 & 0\\
0 & 1+2^{m}
\end{pmatrix};\\
\gamma(a^{i}) & =\begin{pmatrix}0\\
i
\end{pmatrix},\hspace{0.2cm}\qquad\gamma(ba^{i})=\begin{pmatrix}1\\
i
\end{pmatrix}\qquad\text{ for }0\le i\le2^{m+1}-1.
\end{align*}
One can easily verify (see also Proposition \ref{prop:2}) that $\gamma$
becomes a bijective right $1$-cocycle with respect to the $\rho$.
\end{example*}
We now prove the main result of this section following \cite[Section 8 and 9]{Byott}
only making necessary changes suited to our case. 
\begin{thm}
\label{thm:additive} The additive structure of a non-cyclic MMC brace
of size $2^{m+2}$ for $m\ge3$ is $\bbZ/2\times\bbZ/2^{m+1}$. 
\end{thm}

\begin{proof}
It now remains to be shown, by the above Corollary and Proposition
\ref{prop: all_possible_additive_groups}, that there is no regular
subgroup of type $M_{2^{m+2}}$ in $\Hol(N)$ if $N=\bbZ/2\times\bbZ/2\times\bbZ/2^{m}$
or $N=\bbZ/2\times\bbZ/2\times\bbZ/2\times\bbZ/2^{m-1}$ for $m\geq3.$ 

At first, we let $N=\bbZ/2\times\bbZ/2\times\bbZ/2^{m}$. To the contrary
suppose $M_{2^{m+2}}=\left\langle X,Y|X^{2^{m+1}}=1=Y^{2},\,\,YXY=X^{2^{m}+1}\right\rangle $
be a regular subgroup of $\Hol(N)$. Write $X$ as 
\[
\left(\begin{array}{cc}
S & v\\
0 & 1
\end{array}\right)
\]
for some $S$ in $\Aut(N)$ and $v\in N$ (written as a column). By
Proposition \ref{prop:the seven relations} we know that $S$ lies
in a Sylow $2$-subgroup of $\Aut(N).$ By conjugating $A$ with some
$C$ in $\Aut(N)$ (and then conjugating $M_{2^{m+2}}$ by $\left(\begin{array}{cc}
C & 0\\
0 & 1
\end{array}\right)\in\Hol(N)$) we can assume that $S({\rm mod}\,2)$ is upper unipotent (see Proposition
\ref{prop:Aut(N) for all possible Ns}). So we can write 
\[
X=\left(\begin{array}{cccc}
1 & a & b & v_{1}\\
0 & 1 & c & v_{2}\\
2^{m-1}d & 2^{m-1}e & \alpha & v_{3}\\
0 & 0 & 0 & 1
\end{array}\right)
\]
where $a,b,c\in\bbZ/2$ , $d,e\in\bbZ/2^{m}$ , $\alpha\in(\bbZ/2^{m})^{\times}$
, $v_{1},v_{2}\in\bbZ/2$, and $v_{3}\in\bbZ/2^{m}$. Now we compute
the powers of $X$ by first multiplying the matrices treating the
entries as integers and then we reduce the first two rows mod $2$
and the third mod $2^{m}$. We obtain 
\[
X^{2}=\left(\begin{array}{cccc}
1 & 0 & ac & av_{2}+bv_{3}\\
0 & 1 & 0 & cv_{3}\\
0 & 2^{m-1}ad & 2^{m-1}(bd+ce)+\alpha^{2} & 2^{m-1}dv_{1}+2^{m-1}ev_{2}+(1+\alpha)v_{3}\\
0 & 0 & 0 & 1
\end{array}\right).
\]
Again repeating the same strategy, we get 
\[
X^{4}=\left(\begin{array}{cccc}
1 & 0 & 0 & 0\\
0 & 1 & 0 & 0\\
0 & 0 & \alpha^{4} & 2^{m-1}acdv_{3}+(1+\alpha)(1+\alpha^{2})v_{3}\\
0 & 0 & 0 & 1
\end{array}\right).
\]
Now since $\alpha\in(\bbZ/2^{m})^{\times}$ so $\alpha$ is represented
by an odd integer, and hence $\alpha^{4}\equiv1({\rm mod}\,4)$. Also
$m>2$ and $(1+\alpha)(1+\alpha^{2})$ is divisible by $4$. Hence
treating $X^{4}$ as a matrix with integer entries we have 
\[
X^{4}\equiv I({\rm mod}\,4).
\]
We claim that $X$ satisfies 
\[
X^{2^{\ell}}\equiv I({\rm mod}\,2^{\ell}),\quad\quad\text{for all }\ell\geq2.
\]
This can be shown easily by an induction argument, which we omit.
Therefore we have shown that $X^{2^{m}}=I$ in $\Hol(N)$ since $m>2,$
which contradicts the fact that order of $X$ is $2^{m+1}$.

Next, we let $N=\bbZ/2\times\bbZ/2\times\bbZ/2\times\bbZ/2^{m-1}$.
As in the previous case, we proceed by contradiction. We suppose $M_{2^{m+2}}$
is a regular subgroup of $\Hol(N)$. Then we can write 
\[
X=\left(\begin{array}{ccccc}
1 & a & b & c & v_{1}\\
0 & 1 & d & e & v_{2}\\
0 & 0 & 1 & f & v_{3}\\
2^{m-2}g & 2^{m-2}h & 2^{m-2}i & \alpha & v_{4}\\
0 & 0 & 0 & 0 & 1
\end{array}\right).
\]
Then we compute $X^{2}$ by treating integer entries and next reduce
the first three rows mod $2$ and the fourth row mod $2^{m-1}$: 
\[
X^{2}=\left(\begin{array}{ccccc}
1 & 0 & ad & ae+bf & av_{1}+bv_{2}+cv_{3}\\
0 & 1 & 0 & df & dv_{3}+ev_{4}\\
0 & 0 & 1 & 0 & fv_{4}\\
0 & 2^{m-2}ga & 2^{m-2}(gb+hd) & \alpha^{2}+2^{m-2}(gc+he+if) & (1+\alpha)v_{4}+2^{m-2}(gv_{1}+hv_{2}+iv_{3})\\
0 & 0 & 0 & 0 & 1
\end{array}\right).
\]
Squaring and reducing again we get 
\[
X^{4}=\left(\begin{array}{ccccc}
1 & 0 & 0 & 0 & adfv_{4}\\
0 & 1 & 0 & 0 & 0\\
0 & 0 & 1 & 0 & 0\\
0 & 0 & 0 & \alpha^{4}+2^{m-2}gadf & (1+\alpha)(1+\alpha^{2})v_{4}+2^{m-2}\ell\\
0 & 0 & 0 & 0 & 1
\end{array}\right)
\]
for some integer $\ell$. Now when $m=3$ we compute 
\begin{align*}
X^{8}= & \left(\begin{array}{ccccc}
1 & 0 & 0 & 0 & 2adfv_{4}\\
0 & 1 & 0 & 0 & 0\\
0 & 0 & 1 & 0 & 0\\
0 & 0 & 0 & \alpha^{8}+4\alpha gadf & (\alpha^{4}+2gadf+1)((1+\alpha)(1+\alpha^{2})v_{4}+2\ell)\\
0 & 0 & 0 & 0 & 1
\end{array}\right)\\
= & \left(\begin{array}{ccccc}
1 & 0 & 0 & 0 & 0\\
0 & 1 & 0 & 0 & 0\\
0 & 0 & 1 & 0 & 0\\
0 & 0 & 0 & 1 & 0\\
0 & 0 & 0 & 0 & 1
\end{array}\right)\qquad(\text{reducing the first row mod \ensuremath{2}, fourth row mod \ensuremath{4}})
\end{align*}
where we observe that in the last equality we have used the fact that
$\alpha$ is odd and so $\alpha^{4}\equiv1(\text{mod }4)$ and $(1+\alpha)(1+\alpha^{2})\equiv0({\rm mod}\,4)$.
Hence when $m=3$ we have $X^{8}=I$ in $\Hol(N)$ which contradicts
that $X\in M_{32}$ has order $16$. But when $m\geq4$ we have $X^{8}\equiv I({\rm mod}\,4)$
and as in the first case an induction argument gives $X^{2^{k}}\equiv1({\rm mod}\,2^{k-1})$
for every integer $k\geq3$. This implies $X^{2^{m}}=I$ in $\Hol(N)$
which contradicts that $X$ has order $2^{m+1}.$ This completes the
proof. 
\end{proof}

\section{When the additive group is $\protect\bbZ/2\times\protect\bbZ/2^{m+1}$}

In the previous section, we showed that the possible non-cyclic additive
structure of an MMC brace is $\bbZ/2\times\bbZ/2^{m+1}$. We first
give a criteria to test for the existence of MMC braces.
\begin{prop}
\label{prop:2} Let $(A,\circ)$ be a group with a presentation \ref{p1}
and $T,S\in\Aut(\bbZ/2\times\bbZ/2^{m+1})$ such that $S^{2^{m}}=T^{2}=I$
and $TST=S$. Further assume $\gamma:(A,\circ)\to\bbZ/2\times\bbZ/2^{m+1}$
be a bijective map such that $\gamma(1)=0$ and satisfies the following
for all $i,j,k$ 
\begin{align}
\gamma(a^{i+j}) & =S^{j}\gamma(a^{i})+\gamma(a^{j})\label{eq:pp1}\\
\gamma(ba^{i}) & =S^{i}\gamma(b)+\gamma(a^{i})\label{eq:pp2}\\
\gamma(a^{i}b) & =T\gamma(a^{i})+\gamma(b)\label{eq:pp3}\\
T\gamma(b) & =-\gamma(b)\label{eq:pp4}\\
\gamma(a^{2^{m}}) & =\begin{pmatrix}0\\
2^{m}
\end{pmatrix}.\label{eq:pp5}
\end{align}
Then 
\[
\rho:(A,\circ)\to\Aut(\bbZ/2\times\bbZ/2^{m+1}),\qquad\rho(a):=S,~\rho(b):=T
\]
defines a right action of $(A,\circ)$ on $\bbZ/2\times\bbZ/2^{m+1}$,
and with respect to this action $\gamma$ becomes bijective $1$-cocyle,
that is for all $x,y\in A$, $\gamma(xy)=\rho(y)(\gamma(x))+\gamma(y)$. 
\end{prop}

\begin{proof}
We have $(S^{-1})^{2^{m+1}}=(T^{-1})^{2}=I$ and $T^{-1}S^{-1}T^{-1}=(S^{-1})^{2^{m}+1}$,
thus $\rho^{-1}$ (where $\rho^{-1}(x):=\rho(x)^{-1}$) defines an
left action. This yields that $\rho$ is a right action. Clearly for
any automorphism $U\in\Aut(\Z/2\times\Z/2^{m+1})$, we have 
\begin{equation}
U\gamma(a^{2^{m}i})=\gamma(a^{2^{m}i}).\label{eq:pp6}
\end{equation}
Using Eqs. \eqref{eq:pp1}-\eqref{eq:pp4}, \eqref{eq:pp6} and presentation
\eqref{p1}, we get 
\begin{align*}
\gamma((ba^{i})(ba^{j})) & =\gamma(a^{2^{m}i+i+j})=S^{2^{m}i}\gamma(a^{i+j})+\gamma(a^{2^{m}i})\\
 & =\gamma(a^{i+j})+\gamma(a^{2^{m}i}),\\
S^{j}T\gamma(ba^{i})+\gamma(ba^{j}) & =S^{j}T\gamma(a^{2^{m}i+i}b)+\gamma(ba^{j})\\
 & =S^{j}T(T(S^{2^{m}i}\gamma(a^{i})+\gamma(a^{2^{m}i}))+\gamma(b))+\gamma(ba^{j})\\
 & =S^{j}\gamma(a^{i})+S^{j}\gamma(a^{2^{m}i})-S^{j}\gamma(b)+S^{j}\gamma(b)+\gamma(a^{j})\\
 & =S^{j}\gamma(a^{i})+\gamma(a^{j})+\gamma(a^{2^{m}i})=\gamma(a^{i+j})+\gamma(a^{2^{m}i}).
\end{align*}
Using the similar arguments, one can show 
\begin{align*}
\gamma((ba^{i})(a^{j})) & =S^{j}\gamma(ba^{i})+\gamma(a^{j})\\
\gamma((a^{i})(ba^{j})) & =S^{j}T\gamma(a^{i})+\gamma(ba^{j}).
\end{align*}
Hence we have obtained $\gamma(xy)=\rho(y)\gamma(x)+\gamma(y)$ for
all $x,y\in A$.
\end{proof}
We now shift our focus to counting the number of MMC braces of a fixed
size. While the adjoint group of a brace can be viewed as a regular
subgroup in the holomorph of its additive group, the representation
of an element $x\in(A,\circ)$ as a matrix containing $\rho(x)$ and
$\gamma(x)$ introduces computational complexities (see Remark \ref{rem: braces as 1-cocycles and regular subgroups of Hol},
Section $3$). To remove these difficulties, we consider the components
of these matrices separately; in other words, we adopt the approach
of viewing braces as bijective $1$-cocycles, a method previously
utilized in \cite{JAlg,Rump20}.

We consider $(A,\circ)$ to be the group with presentation \ref{p1}.
We denote the bijective $1$-cocycle by $\gamma:(A,\circ)\to\bbZ/2\times\bbZ/2^{m+1}$
and by the proof of the Proposition \ref{prop:Aut(N) for all possible Ns},
the right action $\rho$ of the generators $a$ and $b$ can be written
by the invertible matrices 
\begin{equation}
\rho(a)=\sigma(a)^{-1}=S:=\begin{pmatrix}1 & y\\
2^{m}x & 1+2z
\end{pmatrix},\qquad\rho(b)=\sigma(b)^{-1}=T:=\begin{pmatrix}1 & v\\
2^{m}u & 1+2w
\end{pmatrix}\label{eq:mat-1}
\end{equation}
with $x,y,u,v\in\bbZ/2$ and $z,w\in\bbZ/2^{m}$.

At first we write down some technical lemmas.
\begin{lem}
Let $k\ge1$ and $\beta:=1+2z$, then we have\emph{ : }
\begin{align}
S^{n}=\begin{cases}
\begin{pmatrix}1 & 0\\
0 & \beta^{4k}
\end{pmatrix} & \text{ if }n=4k\\
\begin{pmatrix}1 & y\\
2^{m}x & \beta^{4k+1}
\end{pmatrix} & \text{ if }n=4k+1\\
\begin{pmatrix}1 & 0\\
0 & 2^{m}xy+\beta^{4k+2}
\end{pmatrix} & \text{ if }n=4k+2\\
\begin{pmatrix}1 & y\\
2^{m}x & 2^{m}xy+\beta^{4k+3}
\end{pmatrix} & \text{ if }n=4k+3.
\end{cases}\label{eq:g11}
\end{align}
\end{lem}

\begin{proof}
This is direct computation.
\end{proof}
\begin{lem}
\label{lem:2} Let $k$ be an integer, then there are integers $k_{1},k_{2}$
such that for $m\ge2$, 
\begin{align*}
(4k+1)^{2^{m}-1}+(4k+1)^{2^{m}-2}+\cdots+(4k+1)+1 & =2^{m+1}k_{1}+2^{m}\\
(4k+3)^{2^{m}-1}+(4k+3)^{2^{m}-2}+\cdots+(4k+3)+1 & =2^{m+1}k_{2}.
\end{align*}
In particular, let $\beta\in(\bbZ/2^{m+1})^{\times}$ \emph{with }$m\ge2$\emph{,}
and $x=\beta^{2^{m}-1}+\cdots+\beta+1$. Then we have \emph{:}
\[
x\equiv\begin{cases}
2^{m}\text{ \emph{mod} }\,2^{m+1} & \text{if }\,\beta\equiv1\text{ \emph{mod} }4,\\
0\text{ \emph{mod} }\,2^{m+1} & \text{if }\,\beta\equiv3\text{ \emph{mod }}4.
\end{cases}
\]
\end{lem}

\begin{proof}
This result follows from an easy induction on $m$, and the fact that
for all $m\ge1$ 
\[
\beta^{2^{m}-1}+\cdots+\beta+1=(1+\beta^{2^{m-1}})(\beta^{2^{m-1}-1}+\cdots+\beta+1).
\]
\end{proof}
\begin{lem}
\label{lem:3} Let $\beta=4k+1$ for some $k\ge0$, then 
\[
(\beta^{n}+\cdots+\beta+1)\not\equiv0~(\text{ mod }2^{m-1})
\]
whenever $1\le n\le2^{m+1}-2$ and $n\neq2^{m}-1,~2^{m-1}-1,~2^{m}+2^{m-1}-1$.
And if $1\le n\le2^{m+1}-2$ and $n\neq2^{m}-1$, then 
\[
(\beta^{n}+\cdots+\beta+1)\not\equiv0~(\text{ mod }2^{m}).
\]
\end{lem}

\begin{proof}
Let $\mathsf{o}(\beta)=2^{r}$ in $(\bbZ/2^{m+1})^{\times}$. If $(\beta^{n}+\cdots+\beta+1)\not\equiv0~(\text{ mod }2^{m-1})$,
then the relation 
\begin{align*}
(\beta^{n+1}-1) & =(\beta-1)(\beta^{n}+\cdots+\beta+1),
\end{align*}
gives $\beta^{n+1}=1$, that is $n=2^{r}s-1$ for some $s$. Now 
\begin{align*}
\beta^{n}+\cdots+\beta+1 & =\beta^{2^{r}s-1}+\beta^{2^{r}s-2}+\cdots+\beta+1\\
 & =\beta^{2^{r}(s-1)}(\beta^{2^{r}-1}+\beta^{2^{r}-2}+\cdots+1)+\beta^{2^{r}(s-1)-1}+\beta^{2^{r}(s-1)-2}+\cdots+1\\
 & =(\beta^{2^{r}-1}+\beta^{2^{r}-2}+\cdots+1)+\beta^{2^{r}(s-2)}(\beta^{2^{r}-1}+\cdots+1)+\beta^{2^{r}(s-2)-1}+\beta^{2^{r}(s-2)-2}+\cdots+1\\
 & =s(\beta^{2^{r}-1}+\beta^{2^{r}-2}+\cdots+\beta+1)\\
 & =s(2^{r+1}k'+2^{r})\qquad\text{ by Lemma }\ref{lem:2}.
\end{align*}
Let $s=2^{l}q$, where $q$ is odd. Thus we obtain 
\begin{align}
2^{r++l+1}k'q+2^{r+l}q=2^{m-1}k'',\label{eq:lem1}
\end{align}
since $2\le2^{r+l}q\le2^{m+1}-1$, modulo $2^{r+l+1}$ of \eqref{eq:lem1}
implies $n=2^{m-1}u-1$ for some $u$, which is not possible.\\
 When $(\beta^{n}+\cdots+\beta+1)\not\equiv0~(\text{ mod }2^{m})$,
the Eq. \eqref{eq:lem1} can be replaced by 
\begin{align*}
2^{r++l+1}k'q+2^{r+l}q=2^{m}k'',
\end{align*}
and modulo $2^{r+l+1}$ implies $r+l=m$, then $n=2^{m}-1$ which
is not possible. 
\end{proof}
We first strengthen Corollary \ref{cor:cor 2.2} in the proposition
below.
\begin{prop}
\label{prop:1} Let $A$ be a non-cyclic MMC brace of size $2^{m+2}$,
where $m\ge3$. Then $a^{2^{m-1}}\in\Soc(A)$. 
\end{prop}

\begin{proof}
This follows from the relation $(1+2z)^{2^{m-1}}=1$ in $(\Z/2^{m+1})^\times$, and using Eq \ref{eq:g11},
one can see that $S^{2^{m-1}}=1$ when $m\ge3$. 
\end{proof}
Now we carry out some reductions. From the relations $T^{2}=I,~TST=S$, we get 
\begin{align}
4w(w+1) & =2^{m}uv\label{eq:g1}\\
4w(1+w)(1+2z) & =2^{m}(uv+uy+xv),\label{eq:g2}
\end{align}
in $\Z/{2^{m+1}}$. The above two equation gives 
\begin{equation}
xv=yu ~\text{ in }\Z/2.
\end{equation}
We assume $\gamma(a)=\begin{pmatrix}p\\
q
\end{pmatrix}\text{ and }\gamma(b)=\begin{pmatrix}r\\
s
\end{pmatrix}.$ Thus the relation $\gamma(b^{2})=0$ gives 
\begin{align}
vs &=0  ~\text{ in }\Z/2\label{eq:g3}\\
2s(1+w) & =2^{m}ur ~\text{ in }\Z/{2^{m+1}}.\label{eq:g4}
\end{align}
Note that 

For $k\ge2$, $\gamma(a^{2})=\begin{pmatrix}yq\\
2^{m}xp+\beta q+q
\end{pmatrix}\text{ and }\gamma(a^{2^{k}})=\begin{pmatrix}0\\
(\beta^{2^{k}-1}+\cdots+\beta+1)q
\end{pmatrix},$ since $\gamma(a^{2^{m}})\neq0$ and $\gamma(a^{2^{m+1}})=1$, thus
\begin{equation}
(\beta^{2^{m}-1}+\cdots+\beta+1)q=2^{m},\qquad\gamma(a^{2^{m}})=\begin{pmatrix}0\\
2^{m}
\end{pmatrix}.\label{eq:g5}
\end{equation}
From the relation $bab=a^{2^{m}+1}$, we obtain 
\begin{align*}
\gamma(bab) & =\gamma(a^{2^{m}+1})\\
T\gamma(ba)+\gamma(b) & =S^{2^{m}}\gamma(a)+\gamma(a^{2^{m}})\\
TS\gamma(b)+T\gamma(a)+\gamma(b) & =\gamma(a)+\gamma(a^{2^{m}}),
\end{align*}
this yields 
\begin{align}
vq+s(y+v) & =0~\text{ in }\Z/2\label{eq:g6}\\
2^{m}(xr+ur+xvs+up+1)+q & =\alpha\beta s+\alpha q+s.\label{eq:g7}
\end{align}
Using $\eqref{eq:g4}$, Eq. \eqref{eq:g7} reduces to 
\begin{equation}
2^{m}(xr+xvs+up+1)=2(wq-sz).\label{eq:g8}
\end{equation}

From Lemma \ref{lem:2} and Eq. \eqref{eq:g5}, we obtain $z$ is
even and $q$ is odd. Therefore, we can assume $\gamma(a)=\begin{pmatrix}0\\
1
\end{pmatrix}.$ The Eqs.\eqref{eq:g5}, \eqref{eq:g6} and \eqref{eq:g8} reduce
to 
\begin{align}
(\beta^{2^{m}-1}+\cdots+\beta+1) & =2^{m} ~\text{ in }\Z/{2^{m+1}}\\
v+s(y+v) & =0 ~\text{ in }\Z/2\label{eq:g12}\\
2^{m}(xr+xvs+1) & =2w-2zs ~\text{ in }\Z/{2^{m+1}}.\label{eq:g13}
\end{align}

Modulo $4$ of Eq. \eqref{eq:g13} gives $w\text{ is even.}$ Hence
Eqs \eqref{eq:g1} and \eqref{eq:g4} reduce to 
\begin{align}
4w & =2^{m}uv\\
2s & =2^{m}ur.\label{eq:f2}
\end{align}
Now $4s=0$ implies $s\in\{0,2^{m-1},2^{m},2^{m}+2^{m-1}\}$, and
observe that $\gamma(ba^{2^{m}})=\begin{pmatrix}r\\
s+2^{m}
\end{pmatrix}$, thus using the transformation $(a,b)\mapsto(a,ba^{2^{m}})$, we
can further assume 
\begin{equation}
s\in\{0,2^{m-1}\}.\label{eq:f}
\end{equation}
Since $z$ is even, \eqref{eq:g12} and \eqref{eq:g13} turn into
\begin{align}
v & =0 ~\text{ in }\Z/2\label{eq:f1}\\
2^{m}(xr+1) & =2w  ~\text{ in }\Z/{2^{m+1}}.\label{eq:g14}
\end{align}
\begin{thm}
\label{thm:1} Let $A$ be a non-cyclic MMC brace of size $2^{m+2}$,
where $m\ge3$. Then the brace structure is either of the form 
\begin{align*}
S & =\begin{pmatrix}1 & y\\
2^{m}x & 1+2z
\end{pmatrix},\qquad T=\begin{pmatrix}1 & 0\\
0 & 1+2^{m}(1+x)
\end{pmatrix};\\
\gamma(a) & =\begin{pmatrix}0\\
1
\end{pmatrix},\qquad\qquad\qquad\gamma(b)=\begin{pmatrix}1\\
0
\end{pmatrix};
\end{align*}
\begin{center}
or 
\par\end{center}
\begin{align*}
S & =\begin{pmatrix}1 & 0\\
2^{m}x & 1+2z
\end{pmatrix},\qquad T=\begin{pmatrix}1 & 0\\
2^{m} & 1+2^{m}(1+x)
\end{pmatrix};\\
\gamma(a) & =\begin{pmatrix}0\\
1
\end{pmatrix},\qquad\qquad\qquad\gamma(b)=\begin{pmatrix}1\\
2^{m-1}
\end{pmatrix};
\end{align*}
where $z\in\mathbb{Z}/2^{m}$ is even and $x,y\in\bbZ/2$. Conversely
these two forms define brace structures on the group $(A,\circ)$
of presentation \eqref{p1}. Moreover, these two types of braces are
non-isomorphic. 
\end{thm}

\begin{proof}
From the various relations (above Proposition \ref{prop:2}), we obtain
that $z$ is even and $\gamma(a)$ can be assumed $\begin{pmatrix}0\\
1
\end{pmatrix}$. Furthermore, from Eq. \eqref{eq:f}, we have $s\in\{0,2^{m-1}\}$.
The relations \eqref{eq:g14}, \eqref{eq:f1}, and \eqref{eq:f2},
yield the first form when $s=0$, and the second form when $s=2^{m-1}$.
The first part is complete.

\vspace{0.2cm}

Notice that $\beta=1+2z\in(\bbZ/2^{m+1})^{\times}$ and $|(\bbZ/2^{m+1})^{\times}|=2^{m}$,
therefore $\beta^{2^{m}}=1$. Since $m\ge2$, from Eq.\eqref{eq:g11},
we obtain $S^{2^{m}}=I$. Thus Eq. \eqref{eq:g14} implies $T^{2}=1$
and $T^{-1}S^{-1}T^{-1}=S^{-1}$, therefore from Proposition \ref{prop:2},
\[
a\mapsto S,\qquad b\mapsto T,
\]
define a right action of $(A,\circ)$ on $\bbZ/2\times\mathbb{Z}/2^{m+1}$.

We consider the first form. Using Eq. \eqref{eq:g11}, the relations
\eqref{eq:pp1} and \eqref{eq:pp2}, we obtain, for $k\ge0$:

\begin{align}
\gamma(a^{4k+4}) & =\begin{pmatrix}0\\
\beta^{4k+3}+\cdots+1
\end{pmatrix},\qquad\gamma(a^{4k+1})=\begin{pmatrix}0\\
\beta^{4k}+\cdots+1
\end{pmatrix},\nonumber \\
\gamma(a^{4k+2}) & =\begin{pmatrix}y\\
\beta^{4k+1}+\cdots+1
\end{pmatrix},\qquad\gamma(a^{4k+3})=\begin{pmatrix}y\\
2^{m}xy+\beta^{4k+2}+\cdots+1
\end{pmatrix}\label{eq:gamma-values}\\
\gamma(ba^{4k+4}) & =\begin{pmatrix}1\\
\beta^{4k+3}+\cdots+1
\end{pmatrix},\qquad\gamma(ba^{4k+1})=\begin{pmatrix}1\\
2^{m}x+\beta^{4k}+\cdots+1
\end{pmatrix},\nonumber \\
\gamma(ba^{4k+2}) & =\begin{pmatrix}y+1\\
\beta^{4k+1}+\cdots+1
\end{pmatrix},\qquad\gamma(ba^{4k+3})=\begin{pmatrix}y+1\\
2^{m}x(1+y)+\beta^{4k+2}+\cdots+1
\end{pmatrix}\nonumber 
\end{align}

Clearly in this case, we have relation \eqref{eq:pp4}, and since
$z$ is even, by Lemma \ref{lem:2}, the $\gamma$ satisfies \eqref{eq:pp5},
this implies $\gamma(1)=0$. Now we will verify \eqref{eq:pp3}, it
is clearly true for $i=0$. For $i=4k+1$, $k\ge0$, using the relations
\eqref{eq:gamma-values}, Lemma \ref{lem:2} and the fact $\beta^{2^{m}}=1$,
we have 
\begin{align*}
T\gamma(a^{i})+\gamma(b) & =\begin{pmatrix}1\\
2^{m}x+\beta^{4k}+\cdots+1+2^{m}
\end{pmatrix}\\
\gamma(a^{i}b) & =\gamma(ba^{2^{m}+4k+1})=\begin{pmatrix}1\\
2^{m}x+\beta^{2^{m}+4k}+\cdots+1
\end{pmatrix}\\
 & =\begin{pmatrix}1\\
2^{m}x+\beta^{2^{m}}(\beta^{4k}+\cdots+1)+\beta^{2^{m}-1}+\cdots+1
\end{pmatrix}\\
 & =\begin{pmatrix}1\\
2^{m}x+\beta^{4k}+\cdots+1+2^{m}
\end{pmatrix};
\end{align*}
Similarly, we can obtain $\gamma(a^{i}b)=T\gamma(a^{i})+\gamma(b)$,
by considering $i=4k+2$ and $i=4k+3$. Thus by Proposition \ref{prop:2},
we are only left with showing $\gamma$ is a bijection. Suppose for
some $1\le j\le i\le2^{m+1}-2$ and $u$, we have 
\[
\beta^{i}+\beta^{i-1}+\cdots+1=2^{m}u+\beta^{j}+\cdots+1,
\]
then $\beta^{i-j+1}+\beta^{i-j}+\cdots+1=2^{m}u$. By Lemma \ref{lem:2}
and Lemma \ref{lem:3}, $i-j+1=2^{m}-1$ and $u=1$. Therefore 
\begin{align*}
\text{if } & j\equiv1~(\text{ mod }4),~\text{ then }i\equiv3~(\text{ mod }4)\\
\text{if } & j\equiv2~(\text{ mod }4),~\text{ then }i\equiv0~(\text{ mod }4)\\
\text{if } & j\equiv3~(\text{ mod }4),~\text{ then }i\equiv1~(\text{ mod }4)\\
\text{if } & j\equiv0~(\text{ mod }4),~\text{ then }i\equiv2~(\text{ mod }4).
\end{align*}
Now since $u=1$, 
\[
\gamma(a^{4k_{1}+4})\neq\gamma(a^{4k_{2}+2}),~\gamma(ba^{4k_{1}+4})\neq\gamma(ba^{4k_{2}+2}),\gamma(a^{4k_{1}+4})\neq\gamma(ba^{4k_{2}+2}),~\gamma(ba^{4k_{1}+4})\neq\gamma(a^{4k_{2}+2}).
\]
Now if $\gamma(a^{4k_{1}+1})=\gamma(a^{4k_{2}+3})$ then $y=0$, but
then $u=1$ implies the second coordinate of these elements are not
equal, which is a contradiction, in this way we can show 
\[
\gamma(a^{4k_{1}+1})\neq\gamma(ba^{4k_{2}+3}),~\gamma(b^{4k_{1}+3})\neq\gamma(ba^{4k_{2}+1}),\gamma(ba^{4k_{1}+1})\neq\gamma(ba^{4k_{2}+3}).
\]
Hence, we get $\gamma$ is a bijection.

Similarly, we can show that the second form also defines brace structures.

Now we show that these two type of braces are not isomorphic. Suppose
the braces for $s=0$ and $s=2^{m-1}$ be $A_{1}$ (the cocycle be
$\gamma_{1}$) and $A_{2}$ (cocycle be $\gamma_{2}$) respectively,
and $f:A_{1}\to A_{2}$ be a brace isomorphism. With respect to the
presentation \eqref{p1}, the generators for $(A_{1},\circ)$ be $a_{1},b_{1}$,
and for $(A_{2},\circ)$ be $a_{2},b_{2}$. The elements the in the
group \eqref{p1} of order $2^{m+1}$ are $a^{i}$ or $ba^{i}$ where
$i$ is odd; the elements of order $2$ are $a^{2^{m}},b,\text{ and }ba^{2^{m}}$;
and $a^{2}\in\langle ba\rangle=\langle ba^{i}\rangle$ for all odd
$i$. Thus $f(b_{1})$ is either $b_{2}$ or $b_{2}a_{2}^{2^{m}}$.
But order of $\gamma_{1}(b_{1})$ is $2$, and the order of both $\gamma_{2}(b_{2})$
and $\gamma_{2}(b_{2}a_{2}^{2^{m}})$ is $4$. Which is not possible. 
\end{proof}
From Proposition \ref{prop:subgrps}, we have information of all normal
subgroups of a modular maximal-cyclic group. Now using Theorem \ref{thm:1},
we exhibit all MMC braces of size $2^{m+2}$ ($m\ge3$) by considering
all possible socles.

\subsection{\emph{If $\boldsymbol{\langle ba\rangle=\protect\Soc(A)}$}}

\label{sub1} Here $TS=1$, that is $S=T$. If $s=2^{m-1}$, then
$x=1$ and $2z=0$. Thus in this case 
\begin{align*}
T=S=\begin{pmatrix}1 & 0\\
1 & 1
\end{pmatrix},\qquad\gamma(a)=\begin{pmatrix}0\\
1
\end{pmatrix},\hspace{0.2cm}\gamma(b)=\begin{pmatrix}1\\
2^{m-1}
\end{pmatrix}.
\end{align*}
And if $s=0$, then 
\begin{align*}
T=S=\begin{pmatrix}1 & 0\\
0 & 1+2^{m}
\end{pmatrix},\qquad\gamma(a)=\begin{pmatrix}0\\
1
\end{pmatrix},\hspace{0.2cm}\gamma(b)=\begin{pmatrix}1\\
0
\end{pmatrix}.
\end{align*}
Thus for this case we have two non-isomorphic braces.

\subsection{\emph{If }$\boldsymbol{\langle b,a^{2^{k}}\rangle=\protect\Soc(A)}\text{\emph{\textbf{ for}} }\boldsymbol{k\ge1}$}

\label{sub2} Since $T=1$, $s=0$ and $x=1$. Thus 
\begin{align*}
S & =\begin{pmatrix}1 & y\\
2^{m} & 1+2z
\end{pmatrix},\qquad T=\begin{pmatrix}1 & 0\\
0 & 1
\end{pmatrix};\\
\gamma(a) & =\begin{pmatrix}0\\
1
\end{pmatrix},\qquad\qquad\qquad\gamma(b)=\begin{pmatrix}1\\
0
\end{pmatrix}.
\end{align*}
By Proposition \ref{prop:1}, $1\le k\le m-1$. Since $\mathsf{o}(1+2^{m+1-k})=2^{k}$
in $(\bbZ/2^{m+1})^{\times}$, for each such $k$, there is a brace.
Therefore in this case we have at least $m-1$ distinct braces. 

\subsection{\emph{If }$\boldsymbol{\langle a^{2^{k}}\rangle=\protect\Soc(A)}$
\emph{for} $\boldsymbol{0\protect\leq k\protect\leq m}$}

\label{sub3} In this case $T\neq1$, thus when $s=0$, $x=0$. Hence
\begin{align*}
S & =\begin{pmatrix}1 & y\\
0 & 1+2z
\end{pmatrix},\qquad T=\begin{pmatrix}1 & 0\\
0 & 1+2^{m}
\end{pmatrix}\\
\gamma(a) & =\begin{pmatrix}0\\
1
\end{pmatrix},\qquad\qquad\qquad\gamma(b)=\begin{pmatrix}1\\
0
\end{pmatrix}.
\end{align*}
And when $s=2^{m-1}$ 
\begin{align*}
S & =\begin{pmatrix}1 & 0\\
2^{m}x & 1+2z
\end{pmatrix},\qquad T=\begin{pmatrix}1 & 0\\
2^{m} & 1+2^{m}(1+x)
\end{pmatrix}\\
\gamma(a) & =\begin{pmatrix}0\\
1
\end{pmatrix},\qquad\qquad\qquad\gamma(b)=\begin{pmatrix}1\\
2^{m-1}
\end{pmatrix}.
\end{align*}
Again we have $1\le k\le m-1$. Same reason like above, for each case
we have brace for all such values of $k$, and by Theorem \ref{thm:1},
these braces are non-isomorphic, hence in this case we have at least
$2(m-1)$ distinct braces.

\subsection{\emph{If }$\boldsymbol{\langle a^{2^{k}}b\rangle=\protect\Soc(A)}$
\emph{for} $\boldsymbol{1\le k\le m}$}

\label{sub4} By Propositions \ref{prop:subgrps} and \ref{prop:1},
$k\le m-2$. For $s=2^{m-1}$, 
\begin{align*}
TS^{2^{k}}=\begin{pmatrix}1 & 0\\
2^{m} & \beta^{2^{k}}+2^{m}(1+x)
\end{pmatrix}\neq1.
\end{align*}
Hence $s=0$, since $T\neq1$, $x=0$, and 
\begin{align*}
TS^{2^{k}}=\begin{pmatrix}1 & 0\\
0 & \beta^{2^{k}}+2^{m}
\end{pmatrix}.
\end{align*}
There always exist a $\beta\equiv1~(\text{ mod }4)$ in $(\bbZ/2^{m+1})^{\times}$
such that $\beta^{2^{k}}+2^{m}=1$ and $\beta^{2^{k-1}}+2^{m}\neq1$,
for example $\beta=1+2^{m-k}$. Therefore, we have at least $m-2$
distinct such braces.

\vspace{0.3cm}

Suppose $f:A_{1}\to A_{2}$ be a brace isomorphism, where $A_{1},A_{2}$
are MMC braces, if $f(a_{1})=a_{2}^{i}$ where $i$ is odd, then the
fact $f(\Soc(A_{1}))=\Soc(A_{2})$ implies all the mentioned braces
are distinct. We can assume $f(a_{1})=b_{2}a_{2}$, then $f(a_{1}^{2k})\in\{a_{2}^{2k},a_{2}^{2^{m}+2k}\}$
for all $k\ge1$. Therefore, the braces obtained in Subsections \ref{sub2}--\ref{sub4}
are distinct. Hence we obtain 
\begin{cor}
\label{cor:final} Number of non-cyclic MMC braces of size $2^{m+2}$
with $m\ge3$ is at least $4m-5$. 
\end{cor}

\begin{rem}
A remarkable fact is that the number of braces associated with dihedral,
semidihedral and generalized quaternion groups stabilizes with increasing
order. We have proved that this pattern does not follow for the modular
maximal-cyclic group case. 
\end{rem}

\section*{Acknowledgment}

The second author is grateful to Professor Wolfgang Rump for introducing
this problem.

\bibliographystyle{amsplain}
\bibliography{metacyclic_manuscript}

\providecommand{\bysame}{\leavevmode\hbox to3em{\hrulefill}\thinspace}
\providecommand{\MR}{\relax\ifhmode\unskip\space\fi MR }
\providecommand{\MRhref}[2]{%
  \href{http://www.ams.org/mathscinet-getitem?mr=#1}{#2}
}
\providecommand{\href}[2]{#2}
\begin{thebibliography}{10}

\bibitem{Angiono2017}
Iv\'an Angiono, C\'esar Galindo, and Leandro Vendramin, \emph{Hopf braces and
  {Y}ang-{B}axter operators}, Proc. Amer. Math. Soc. \textbf{145} (2017),
  no.~5, 1981--1995. \MR{3611314}

\bibitem{Bardakov2020}
Valeriy~G. Bardakov, Mikhail~V. Neshchadim, and Manoj~K. Yadav, \emph{Computing
  skew left braces of small orders}, Internat. J. Algebra Comput. \textbf{30}
  (2020), no.~4, 839--851. \MR{4113853}

\bibitem{Berko}
V.~G. Berkovi\v{c}, \emph{Groups of order {$p\sp{n}$} that admit an
  automorphism of order {$p\sp{n-1}$}}, Algebra i Logika \textbf{9} (1970),
  3--8. \MR{283084}

\bibitem{Byott}
Nigel~P. Byott and Fabio Ferri, \emph{On the number of quaternion and dihedral
  braces and {H}opf-{G}alois structures}, J. Algebra \textbf{665} (2025),
  72--102. \MR{4830431}

\bibitem{Caranti2006}
A.~Caranti, F.~Dalla~Volta, and M.~Sala, \emph{Abelian regular subgroups of the
  affine group and radical rings}, Publ. Math. Debrecen \textbf{69} (2006),
  no.~3, 297--308. \MR{2273982}

\bibitem{MSC23}
Marco Castelli, Marzia Mazzotta, and Paola Stefanelli, \emph{Simplicity of
  indecomposable set-theoretic solutions of the {Y}ang-{B}axter equation},
  Forum Math. \textbf{34} (2022), no.~2, 531--546. \MR{4388351}

\bibitem{Catino2016}
Francesco Catino, Ilaria Colazzo, and Paola Stefanelli, \emph{Regular subgroups
  of the affine group and asymmetric product of radical braces}, J. Algebra
  \textbf{455} (2016), 164--182. \MR{3478858}

\bibitem{COK}
F.~Ced\'o and J.~Okni\'nski, \emph{Constructing finite simple solutions of the
  {Y}ang-{B}axter equation}, Adv. Math. \textbf{391} (2021), Paper No. 107968,
  40. \MR{4300920}

\bibitem{IYB-grp}
Ferran Ced\'o, Eric Jespers, and \'Angel del R\'io, \emph{Involutive
  {Y}ang-{B}axter groups}, Trans. Amer. Math. Soc. \textbf{362} (2010), no.~5,
  2541--2558. \MR{2584610}

\bibitem{Cedo2014}
Ferran Ced\'o, Eric Jespers, and Jan Okni\'nski, \emph{Braces and the
  {Y}ang-{B}axter equation}, Comm. Math. Phys. \textbf{327} (2014), no.~1,
  101--116. \MR{3177933}

\bibitem{Childs2013}
Lindsay~N. Childs, \emph{Fixed-point free endomorphisms and {H}opf {G}alois
  structures}, Proc. Amer. Math. Soc. \textbf{141} (2013), no.~4, 1255--1265.
  \MR{3008873}

\bibitem{Crespo2023a}
Teresa Crespo, Daniel Gil-Mu\~noz, Anna Rio, and Montserrat Vela,
  \emph{Inducing braces and {H}opf {G}alois structures}, J. Pure Appl. Algebra
  \textbf{227} (2023), no.~9, Paper No. 107371, 16. \MR{4559373}

\bibitem{Crespo2023}
\bysame, \emph{Left braces of size {$8p$}}, J. Algebra \textbf{617} (2023),
  317--339. \MR{4513787}

\bibitem{ESS99}
Pavel Etingof, Travis Schedler, and Alexandre Soloviev, \emph{Set-theoretical
  solutions to the quantum {Y}ang-{B}axter equation}, Duke Math. J.
  \textbf{100} (1999), no.~2, 169--209. \MR{1722951}

\bibitem{Fea2012}
S.~C. Featherstonhaugh, A.~Caranti, and L.~N. Childs, \emph{Abelian {H}opf
  {G}alois structures on prime-power {G}alois field extensions}, Trans. Amer.
  Math. Soc. \textbf{364} (2012), no.~7, 3675--3684. \MR{2901229}

\bibitem{GV17}
L.~Guarnieri and L.~Vendramin, \emph{Skew braces and the {Y}ang-{B}axter
  equation}, Math. Comp. \textbf{86} (2017), no.~307, 2519--2534. \MR{3647970}

\bibitem{hall73}
Marshall Hall, Jr., \emph{The theory of groups}, Chelsea Publishing Co., New
  York, 1976, Reprinting of the 1968 edition. \MR{414669}

\bibitem{hillar2006}
Christopher~J. Hillar and Darren~L. Rhea, \emph{Automorphisms of finite abelian
  groups}, Amer. Math. Monthly \textbf{114} (2007), no.~10, 917--923.
  \MR{2363058}

\bibitem{BBMS}
Arpan Kanrar and Wolfgang Rump, \emph{A decomposition problem for involutive
  solutions to the {Y}ang-{B}axter equation}, Bull. Belg. Math. Soc. Simon
  Stevin \textbf{31} (2024), no.~5, 688--702. \MR{4842840}

\bibitem{JAlg}
\bysame, \emph{Braces with adjoint group of maximal nilpotency class}, J.
  Algebra \textbf{664} (2025), 328--343. \MR{4829401}

\bibitem{Rump07}
Wolfgang Rump, \emph{Braces, radical rings, and the quantum {Y}ang-{B}axter
  equation}, J. Algebra \textbf{307} (2007), no.~1, 153--170. \MR{2278047}

\bibitem{RumpCyc}
\bysame, \emph{Classification of cyclic braces}, J. Pure Appl. Algebra
  \textbf{209} (2007), no.~3, 671--685. \MR{2298848}

\bibitem{Rump20}
\bysame, \emph{Classification of the affine structures of a generalized
  quaternion group of order {$\geqslant32$}}, J. Group Theory \textbf{23}
  (2020), no.~5, 847--869. \MR{4141382}

\bibitem{mpl2}
\bysame, \emph{Classification of non-degenerate involutive set-theoretic
  solutions to the {Y}ang-{B}axter equation with multipermutation level two},
  Algebr. Represent. Theory \textbf{25} (2022), no.~5, 1293--1307. \MR{4483006}

\bibitem{zass58}
H.~J. Zassenhaus, \emph{The theory of groups}, 2nd ed., Chelsea, New York,
  1958.

\end{thebibliography}

\end{document}